\documentclass[twoside,12pt]{amsart}

\usepackage[foot]{amsaddr}

\author{Matthew D. Kvalheim}
\address[Kvalheim, Gustafson]{School of Engineering and Applied Science, University of Pennsylvania, Philadelphia, PA 19104, USA, \texttt{kvalheim@seas.upenn.edu}, \texttt{pgu@seas.upenn.edu}}

\author{Paul Gustafson}

\author{Samuel A. Burden}
\address[Burden]{Department of Electrical \& Computer Engineering, University of Washington, Seattle, WA 98195, USA, \texttt{sburden@uw.edu}}

\title{A pasting lemma for Lipschitz functions}

\usepackage{import} 

\usepackage[utf8]{inputenc}
\usepackage[T1]{fontenc}
\usepackage{lmodern}
\usepackage{amsmath}
\usepackage{amssymb}
\usepackage{amsthm}
\usepackage{stmaryrd} 
\usepackage{mathtools}
\usepackage{tikz-cd}
\usepackage{subfig}

\usepackage{thmtools}
\usepackage{thm-restate} 

\newcommand{\concept}[1]{\textbf{#1}}

\newcommand{\R}{\mathbb{R}}

\newcommand{\slot}{\,\cdot\,} 

\newcommand{\cl}{\textnormal{cl}}

\newcommand{\Lip}[1]{\textnormal{Lip}(#1)}

\newcommand{\distNA}{d}  
\newcommand{\dist}[2]{\distNA( #1, #2)}
\newcommand{\distSNA}{\delta}  
\newcommand{\distS}[2]{\distSNA( #1, #2)} 

\DeclarePairedDelimiter\norm{\lVert}{\rVert}

\newtheorem{Th}{Theorem}

\newtheorem{Prop}{Proposition}

\theoremstyle{definition}

\newtheorem*{Def*}{Definition}

\newtheorem{Rem}{Remark}

\newcommand{\thistheoremname}{}
\newtheorem*{genericthm}{\thistheoremname}
{\renewcommand{\thistheoremname}{Theorem~\ref{#1}$'$}%
	\begin{genericthm}}
	{\end{genericthm}}

\usepackage{marvosym}
\usepackage[
hypertexnames,%
citecolor=blue,%
colorlinks=true,%
linkcolor=red%
]{hyperref}
\usepackage[all]{hypcap}

\begin{document}
	
\maketitle
\begin{abstract}
We give a necessary and sufficient condition ensuring that any function which is separately Lipschitz on two fixed compact sets is Lipschitz on their union.
\end{abstract}	
	
The pasting \cite[Thm~18.3]{munkres2000topology} (or gluing \cite[Lem.~3.23]{lee2010topological}) lemma asserts that a function which is separately continuous on two closed sets is continuous on their union.
To quote \cite[p.~7]{weaver2018lipschitz}: \begin{quote}
``...there is no gluing lemma for Lipschitz functions. A function
which is separately Lipschitz on two closed sets, even on two compact sets,
need not be Lipschitz on their union.''
\end{quote}
A simple example of this phenomenon is given in \cite[Ex.~1.12]{weaver2018lipschitz}.
However, in this note we give a necessary and sufficient condition for a pair of compact sets to satisfy the Lipschitz analog of the pasting lemma for any separately Lipschitz function on the pair (Theorem~\ref{th:lipschitz-pasting}).
We then show that this condition is local (Prop.~\ref{prop:lp-local}) and give a broad class of examples for which it is satisfied (Prop.~\ref{prop:manifold-transverse}).

Let $(X,\distNA)$ and $(Y,\distSNA)$ be  metric spaces.
Given a map $f\colon X\to Y$, we define
$$\Lip{f}\coloneqq \sup_{x_1\neq x_2}\frac{\distS{f(x_1)}{f(x_2)}}{\dist{x_1}{x_2}} \in [0,\infty]$$
and say that $f$ is \concept{Lipschitz} if $\Lip{f}<\infty$.
We say that a pair $(A,B)$ of compact subsets of $X$ are \concept{LP} (for ``Lipschitz paste-able'') if  
they are disjoint, or if they are not disjoint and there is $K > 0$ such that
\begin{equation}\label{eq:lp}
\forall a\in A, b\in B\colon \inf_{x\in A\cap B}[\dist{a}{x} +  \dist{x}{b}]\leq K \dist{a}{b}.
\end{equation}
By taking $x = a$ and $K=1$, we see that $(A,B)$ are LP when $A = B$.
We say that the \concept{Lipschitz pasting lemma holds} for $(A,B)$ if, for any metric space $(Y,\distSNA)$ and map $f\colon A\cup B\to Y$, $f$ is Lipschitz if $f|_A$ and $f|_B$ are Lipschitz.

\begin{Th}\label{th:lipschitz-pasting}
Let $(A,B)$ be a pair of compact subsets of a metric space $(X,\distNA)$.
Then the Lipschitz pasting lemma holds for $(A,B)$ if and only if $(A,B)$ are LP.
\end{Th}   

\begin{proof}
Assume that $(A,B)$ are LP and let $f\colon A\cup B\to Y$ be any map to a metric space $(Y,\distSNA)$ such that $f|_A$ and $f|_B$ are Lipschitz.
We need to show that $f$ is Lipschitz.
If $A\cap B = \varnothing$ then, since $A$ and $B$ are compact and hence separated by a positive distance, $f$ is Lipschitz.
If instead $A\cap B\neq \varnothing$, it suffices to exhibit $L>0$ such that
\begin{equation}\label{eq:lipschitz-pasting-desired}
\forall a\in A, b\in B\colon \distS{f(a)}{f(b)}\leq L \dist{a}{b}
\end{equation}
since this will imply that $f$ is Lipschitz with $\Lip{f}\leq \max(L,\Lip{f|_A},\Lip{f|_B})$.
Defining $L_0\coloneqq \max(\Lip{f|_A}, \Lip{f|_B})$, the triangle inequality implies that
\begin{align*}
\forall x\in A\cap B\colon \distS{f(a)}{f(b)}\leq \distS{f(a)}{f(x)} + \distS{f(x)}{f(b)}\leq L_0\cdot [\dist{a}{x} + \dist{x}{b}]. 
\end{align*}
Taking the infimum over all $x\in A\cap B$ yields
\begin{align}\label{eq:lipschitz-pasting-2}
\distS{f(a)}{f(b)}\leq L_0\inf_{x\in A\cap B}[\dist{a}{x} + \dist{x}{b}]\leq L_0K \dist{a}{b},
\end{align}
where the second inequality follows from \eqref{eq:lp}.
Thus, \eqref{eq:lipschitz-pasting-desired} is satisfied with $L\coloneqq L_0 K$.

Conversely, assume that the Lipschitz pasting lemma holds for $(A,B)$. 
If $A\cap B = \varnothing$ then $(A,B)$ are LP, so assume that $A\cap B\neq \varnothing$.  
We define $Y\coloneqq A\cup B$ and observe that $\delta:Y\times Y\rightarrow [0,\infty)$ defined for all $x_1,x_2\in Y$ by
\begin{equation}\label{eq:special-metric}
\begin{split}
  & \distS{x_1}{x_2}\coloneqq \begin{cases}
\dist{x_1}{x_2}, &x_1,x_2\in A \textnormal{ or } x_1,x_2\in B\\
\inf_{x\in A\cap B}[\dist{x_1}{x} + \dist{x}{x_2}], & \textnormal{otherwise}
\end{cases}
\end{split}
\end{equation}
is a metric on $Y$. (Symmetry is immediate,  positive-definiteness follows since $\distNA \leq \distSNA$, and the triangle inequality is verified by checking cases).
Let $f\colon (A\cup B,\distNA|_{(A\cup B)^2})\to (Y,\distSNA)$ be the identity when viewed as a map of sets.
Then for all $a_1,a_2\in A$, $b_1,b_2\in B$
$$\distS{f(a_1)}{f(a_2)}= \distS{a_1}{a_2} = \dist{a_1}{a_2} \quad \textnormal{and} \quad \distS{f(b_1)}{f(b_2)}=\distS{b_1}{b_2}=\dist{b_1}{b_2}$$ 
by the definition of $\distSNA$, so $f|_A$ and $f|_B$ are Lipschitz.
Since the Lipschitz pasting lemma holds for $(A,B)$, it follows that $f$ is Lipschitz.
Hence there is $K > 0$ such that, for all $a\in A$ and $b\in B$,
\begin{equation}
\inf_{x\in A\cap B}[\dist{a}{x} + \dist{x}{b}] = \distS{a}{b} = \distS{f(a)}{f(b)} \leq K \dist{a}{b}.
\end{equation}
It follows that $(A,B)$ are LP, as desired.
\end{proof}

The following result shows that the property of being LP is local.
\begin{Prop}[LP is a local property]\label{prop:lp-local}
Let $(A,B)$ be a pair of compact subsets of a metric space $(X,\distNA)$.
Then $(A,B)$ are LP if and only if for every $y\in A\cap B$ there is an open neighborhood $U$ of $y$ and $K > 0$ such that 
\begin{equation}\label{eq:lp-local}
\forall a\in A\cap U, b\in B\cap U\colon \inf_{x\in A\cap B\cap U}[\dist{a}{x} +  \dist{x}{b}]\leq K \dist{a}{b}.
\end{equation}
\end{Prop}
\begin{proof}
By taking $U = X$ for each $y\in A\cap B$, it is immediate that $(A,B)$ satisfy the property from the proposition if $(A,B)$ are LP.
 
Conversely, assume that for every $y\in A\cap B$ there is an open neighborhood  $U$ of $y$ and constant $K> 0$ such that \eqref{eq:lp-local} holds.
If $A\cap B = \varnothing$ then $(A,B)$ are LP, so assume that $A\cap B \neq \varnothing$.
Since $A\cap B$ is compact there are finitely many $y_1,\ldots, y_n\in A\cap B$ with corresponding $K_1,\ldots, K_n$ and $U_1,\ldots, U_n$ satisfying $A\cap B \subset \bigcup_i U_i$.
Define the compact set $C\coloneqq (A\times B) \setminus \Delta$ where $\Delta \coloneqq \bigcup_i U_i \times U_i$.  
Then
\begin{align*}
  \sup_{\substack{a\in A, b\in B\\ a\neq b}} & \frac{\inf_{x\in A\cap B}\dist{a}{x}+\dist{x}{b}}{\dist{a}{b}} \\
  \leq &  \sup_{(a,b)\in C}\frac{\inf_{x\in A\cap B}\dist{a}{x}+\dist{x}{b}}{\dist{a}{b}} + \sup_{\substack{(a,b) \in (A \times B) \cap \Delta \\ a\neq b}}\frac{\inf_{x\in A\cap B}\dist{a}{x}+\dist{x}{b}}{\dist{a}{b}} \\
  \leq &  \sup_{(a,b)\in C}\frac{\inf_{x\in A\cap B}\dist{a}{x}+\dist{x}{b}}{\dist{a}{b}} + \max_i \sup_{\substack{a\in A \cap U_i, b \in B \cap U_i \\ a \neq b}}\frac{\inf_{x\in A\cap B \cap U_i }\dist{a}{x}+\dist{x}{b}}{\dist{a}{b}},
\end{align*}
where the first and second inequalities use the fact that the supremum over a set equals the maximum of the suprema over any finite cover, and the second inequality additionally uses the fact that the infimum over a set is no larger than the infimum over any subset.
The second summand in the final line is bounded by $\max_i K_i$, and the first is bounded since  $\distNA$ is bounded away from zero on the compact set $C$ (which is disjoint from the diagonal $\{(z,z)\colon z\in A\cup B\}$).  It follows that $(A,B)$ are LP.
\end{proof}

The following result establishes a class of examples arising in differential geometry \cite{lee2013smooth,lee2018riemannian} for which the LP property holds.
\begin{Prop}\label{prop:manifold-transverse}
Let $X$ be a connected smooth Riemannian manifold with $\distNA\colon X\times X\to [0,\infty)$ the induced path-length metric.
Let $A,B\subset X$ be compact boundaryless $C^1$ embedded submanifolds with transverse intersection.
Then $(A,B)$ are LP.
\end{Prop}
\begin{proof}
If $A\cap B = \varnothing$ then $(A,B)$ are LP, so assume that $A\cap B\neq \varnothing$.
Fix an arbitrary $y\in A\cap B$.
By Prop.~\ref{prop:lp-local} it suffices to find an open neighborhood $U$ of $y$ and $K > 0$ such that \eqref{eq:lp-local} holds.

Let $n$ be the dimension of $X$.
Transversality implies the existence of precompact open sets $U$, $V$, $W$ and a $C^1$ diffeomorphism $\varphi = (u,v,w)\colon W\to \R^{n}$ such that $V$ is geodesically convex \cite[p.~166]{lee2018riemannian}, $\varphi(U)$ is an open Euclidean ball, $$y\in U\subset \cl(U)\subset V\subset \cl(V) \subset W,$$ $\varphi(y)=0$, $\varphi(A\cap W) = \{u=0\}$ and $\varphi(B\cap W)= \{v=0\}$.
Let $\norm{\slot}$ denote the standard Euclidean norm on $\R^n$.
For any $a\in A \cap W$ and $b\in B\cap W$, subadditivity of $\sqrt{\slot}$ and the Cauchy-Schwarz inequality imply that
\begin{equation*}
\begin{split}
\norm{\varphi(a)-\varphi(b)}&= \sqrt{\norm{v(a)}^2+\norm{u(b)}^2+\norm{w(a)-w(b)}^2}\\
&\leq \norm{v(a)}+\norm{u(b)}+\norm{w(a)-w(b)}\\
&= (1,1,1)\cdot (\norm{v(a)},\norm{u(b)}, \norm{w(a)-w(b)} )\\
&\leq \sqrt{3} \norm{\varphi(a)-\varphi(b)}.
\end{split}
\end{equation*}
Hence for any $a\in A\cap U$, $b\in B\cap U$
\begin{equation}\label{eq:manifold-transverse-2}
\begin{split}
\inf_{x\in A\cap B\cap U}\norm{\varphi(a)-\varphi(x)} + \norm{\varphi(x)-&\varphi(b)} \leq \norm{v(a)} + \norm{u(b)}\\ 
&+ \inf_{x\in A\cap B\cap U}\norm{w(a)-w(x)} + \norm{w(x)-w(b)}\\
&= \norm{v(a)} + \norm{u(b)} + \norm{w(a)-w(b)}\\
& \leq \sqrt{3} \norm{\varphi(a)-\varphi(b)}
\end{split}
\end{equation}
(The equality follows from the fact that the second infimum is attained when $w(x)$ is any point on the line segment joining $w(a)$ and $w(b)$, which is in turn a consequence of elementary geometric considerations.)
On the other hand, $\varphi(\cl(V))$ is compact, and all Riemannian metrics are uniformly equivalent on compact sets.
This together with geodesic convexity of $V\supset U$ implies the existence of $K_0>0$ such that, for all $a\in A\cap U$ and $b\in B\cap U$, 
\begin{equation}\label{eq:manifold-transverse-3}
\begin{split}
\frac{1}{K_0}\dist{a}{b}&\leq \norm{\varphi(a)-\varphi(b)}\leq K_0 \dist{a}{b}
\end{split}
\end{equation}
Taking $K\coloneqq K_0^2 \sqrt{3}$, \eqref{eq:lp-local} follows from \eqref{eq:manifold-transverse-2} and \eqref{eq:manifold-transverse-3}.
By the first paragraph of the proof we are done.
\end{proof}

\section*{Acknowledgments}
Kvalheim was supported in part by the Army Research Office (ARO) under the SLICE Multidisciplinary University Research Initiatives (MURI) Program, award W911NF1810327. 
Gustafson was supported in part by UATL 10601110D8Z, a LUCI Fellowship held by Jared Culbertson, granted by the Basic Research Office of the US Undersecretary of Defense for Research and Engineering.
Kvalheim and Gustafson were also supported in part by ONR N000141612817, a Vannevar Bush Faculty Fellowship held by Daniel E. Koditschek,  sponsored by the Basic Research Office of the Assistant Secretary of Defense for Research and Engineering.
Burden was supported by National Science Foundation Cyber-Physical Systems Program Award \#1836819.

\bibliographystyle{amsalpha}
\bibliography{ref}
	
\end{document}